 \documentclass[11pt,english,a4paper]{smfart}
  \usepackage[colorlinks, pagebackref, linkcolor=red, citecolor=blue]{hyperref}
 \usepackage{amsmath,amssymb}
 \usepackage{color}
 \usepackage{mathrsfs}
 \usepackage{graphicx}
 \usepackage{lineno}
 
\renewcommand{\thesubsection}{\thesection .\Alph{subsection}}
 
 \usepackage{calc}
\makeatletter
 
\newlength{\temp@wc@width}

\newlength{\temp@wc@height}

\newcommand{\widecheck}[1]{%
\setlength{\temp@wc@width}{\widthof{$#1$}}%
\setlength{\temp@wc@height}{\heightof{$#1$}}%
#1\hspace{-\temp@wc@width}%
\raisebox{\temp@wc@height+2pt}[\heightof{$\widehat{#1}$}]%
{\rotatebox[origin=c]{180}{\vbox to 0pt{\hbox{$\widehat{\hphantom{#1}}$}}}}%
}

\makeatother
 
 \def\M {\mbox{\boldmath$ M$}}
 \def\kv{ \mbox{\boldmath$ k$}}

 \def\kvn{{\mbox{\footnotesize \boldmath$ k \in \M^j$}\atop \mbox
{\footnotesize \boldmath$|k|$=\it n}}}
 \def\kvm{{\mbox{\footnotesize \boldmath$ k \in \M^j$}\atop \mbox
{\footnotesize {\boldmath$|k|$}=  $\it\mu$}}}
 \def\kvl{{\mbox{\footnotesize \boldmath$ k \in \M^j$}\atop \mbox
{\footnotesize {\boldmath$|k|$}=$\ell$}}}

 \numberwithin{equation}{section}
 \theoremstyle{plain}
 \newtheorem{theorem}{Theorem}[section]
 \newtheorem{lemma}[theorem]{Lemma}
  \newtheorem{problem}[theorem]{Problem}

 \newtheorem{corollary}[theorem]{Corollary}
 \newtheorem{proposition}[theorem]{Proposition}

   \theoremstyle{definition}
\newtheorem{definition} [theorem] {Definition}

 \newcommand{\zit}[1]{(\ref{#1})}

   \def\dint{\int\hspace{-5pt}\int}

 \def\supp{\operatorname{supp}}
 \def\bs{\boldsymbol}

 \def\R{ \mathbb R}

 \def\L{ \mathcal L}

 \def\D{{ \mathbb D}}
  \def\K{{ \mathbb K}}
 \def\C{{ \mathbb C}}
 
 \def\N{{ \mathbb N}}

  \def\F{{ \mathscr F}}
 
 \def\e{\varepsilon}
 \def\dbar{\ov\partial}
 \def\sp {\quad}
 \def\ssc{\scriptscriptstyle}
 \def\sc{\scriptstyle}
 
 \def\bs{\boldsymbol}
 \def\dis{\displaystyle}
 \def\union{\cup}
 \def\Union{\bigcup}
 \def\inter{\cap}
 \def\Inter{\bigcap }
 \def\ov{\overline}

 \def\ss{\subseteq}
 \def\emp{\emptyset}

 \def\oh{{\scriptstyle{\mathcal O}}}
 
 \def\buildrel#1_#2^#3{\mathrel{\mathop{\kern 0pt#1}\limits_{#2}^{#3}}}
 
\begin{document}

 \title [Corona theorems]{Corona-type theorems and division in some function algebras
  on planar domains}


 \author{Raymond Mortini}
  \address{
   \small Universit\'{e} de Lorraine\\
\small D\'{e}partement de Math\'{e}matiques et  
Institut \'Elie Cartan de Lorraine,  UMR 7502\\
\small Ile du Saulcy\\
 \small F-57045 Metz, France} 
 \email{mortini@univ-metz.fr}

\author{Rudolf Rupp}
\address{ Fakult\"at Allgemeinwissenschaften\\
\small Georg-Simon-Ohm-Hochschule N\"urnberg\\
\small Kesslerplatz 12\\
\small D-90489 N\"urnberg, Germany
}
\email  {Rudolf.Rupp@ohm-hochschule.de}


 
 \begin{abstract}
Let  $A$ be an algebra of bounded smooth functions on the interior of a compact set  in the plane.
We study the following problem:  if $f,f_1,\dots,f_n\in A$ satisfy  $|f|\leq \sum_{j=1}^n |f_j|$,
does there exist $g_j\in A$ and a  constant $N\in\N$ such that  $f^N=\sum_{j=1}^n g_j f_j$?
A prominent role in our proofs is played by a new space, $C_{\dbar, 1}(K)$, 
which  we call the algebra of $\dbar$-smooth functions.

In the case $n=1$, a complete solution is given for the algebras $A^m(K)$ of functions 
holomorphic in $K^\circ$ and whose first $m$-derivatives extend continuously to $\ov{K^\circ}$.
This necessitates the introduction of a special class of compacta, the so-called locally L-connected sets.

We also present another  constructive proof of  the Nullstellensatz for $A(K)$, that is only  based
on elementary $\dbar$-calculus and  Wolff's method.

  \end{abstract}

  \maketitle

 \centerline {\small\the\day.\the \month.\the\year} \medskip
 \section*{Introduction}

 Our paper is motivated by the following problem.
 \begin{problem} \label{N=3}
 Let $A(K)$ be the algebra of  all complex-valued functions that are
 continuous on the compact  set $K\ss\C$ and holomorphic  in the interior $K^\circ$ of $K$.
 Given $f,f_1,\dots, f_n\in A(K)$ satisfying 
 \begin{equation}\label{funda}
 |f|\leq \sum_{j=1}^n |f_j|,
 \end{equation} 
 does there exist a power $N$ such that $f^N$ belongs to the ideal $I_{A(K)}(f_1,\dots, f_n)$ generated by 
 the $f_j$ in $A(K)$?
  \end {problem}
 
 In view of Wolff's  result for the algebra of bounded holomorphic  functions in the open unit
 disk $\D$  (see \cite{ga}), the expected value for $N$ seems to be 3. We are unable to confirm this;
 our intention therefore is to present sufficient conditions that guarantee the existence of such
 a constant that may depend on the $n$-tuple $(f_1,\dots, f_n)$.
 If $f$ is zero-free on $K$ then, by the   classical Nullstellensatz  for $A(K)$, 
 $f=\sum_{j=1}^n f_jg_j$ for some $g_j\in A(K)$, whenever the $f_j$ satisfy \zit{funda}.
 We present a constructive proof of this assertion by using $\dbar$-calculus.
 A different constructive  proof, indirectly based on properties of the Pompeiu operator 
 $$Pf(z):=\dint_K \frac{f(\xi)}{\xi-z}\; d\sigma(\xi)$$
 and elementary approximation theory,  was recently given by the authors in \cite{mr}.
 
 To achieve our goals,  we consider various algebras of smooth functions on $K$.  Here is our setting.
 First we present some notation and introduce, apart from $A(K)$ that was defined above, 
  the spaces we are dealing with.
 
As usual, if $f_x$ and $f_y$ are the partial derivatives of $f$, then 
 $$\mbox{$\partial f= (1/2)(f_x-if_y)$ and $\dbar f=(1/2)(f_x+if_y)$}$$
  denote the Wirtinger derivatives of $f$.
 Let $\K$ be either $\R$ or $\C$.

 \begin{definition}\hfill
 
 \begin{enumerate}
 \item [(1)]  $C(K,\K)$ is the set of all  continuous, $\K$-valued functions on $K$.
 If $\K=\C$, then we also write $C(K)$ instead of $C(K,\C)$.
  \item [(2)] $C^m(K)$ is the set of all $f\in C(K)$ that are  $m$-times continuously 
 differentiable on $K^\circ$ and such that each of the partial derivatives up to the order $m$
  extends  continuously to $K$;
 \item  [(3)]  $A^m(K):= A(K)\inter C^m(K)$.
 \end{enumerate}
 Of course $A^m(K)$ is the set of all functions $f\in A(K)$ such that 
 $f^{(j)}|_{K^\circ}$ extends continuously to $K$ for $j=1,\dots,m$.
  Finally, we introduce the following  space of $\dbar$-smooth functions.
 This space will play an important role when solving $\dbar$-equations.
 \begin{enumerate}
\item [(4)] $C_{\dbar,1}(K)$ is the set of all $f\in C(K)$
 continuously differentiable in  $K^\circ$ with $\dbar f\in C^1(K^\circ)$
 and such that $\dbar f$ admits a continuous extension to $K$.
\end{enumerate}
\end{definition}
 
 The set of continuous  (respectively infinitely often differentiable) 
 functions on $\C$  with compact 
 support  is denoted  by $C_c(\C)$ (respectively $C^\infty_c(\C)$). 
 If $f\in C(K,\K)$, then $||f||_\infty=\sup_{z\in K}|f(z)|$ and
 $$Z(f)=\{z\in K: f(z)=0\}$$
 is the {\it zero set} of $f$.
 
 If $A$ is a commutative unital algebra with unit element denoted by $1$,  then 
 the ideal generated by $f_1,\dots, f_n\in A$ is denoted by $I_A(f_1,\dots,f_n)$;
 that is
 $$I_A(f_1,\dots,f_n)=\biggl\{\sum_{j=1}^n g_j f_j: g_j\in A\biggr\}.$$

 \section{The algebra of $\dbar$-smooth functions}

 If we endow $C_{\dbar,1}(K)$ with the pointwise operations $(+, \bullet, \cdot_{s})$, then 
 it becomes an algebra; just note that for $f,g\in C_{\dbar,1}(K)$,
 $\dbar(fg)=(\dbar f) g+f \dbar g$.   The most important subalgebra of  $C_{\dbar,1}(K)$ is $A(K)$.
 It is obvious that $C_{\dbar,1}(K)$ contains all the polynomials in the real variables $x$ and $y$,
  or what is equivalent, all polynomials in $z$ and $\ov z$.  In particular 
  $C_{\dbar,1}(K)$ is uniformly dense
  in $C(K)$. Note  that  the class of real-valued functions $u$ in  $C_{\dbar,1}(K)$ coincides  with
 $C^2(K^\circ,\R)\inter C^1(K,\R)$, since   $2\dbar f=u_x +i u_y \in C^1(K^\circ)\inter C(K)$.
  We point out
 that for $f\in C_{\dbar,1}(K)$, the function $\partial f:K^\circ\to\C$  may not even be bounded. 
 Consider for example on the unit disk $\D$
  the standard holomorphic  branch of the function $ \sqrt{1-z}$. Here 
  $$\partial f(z)=f '(z)= \frac{1}{2}(1-z)^{-1/2},$$
   which is unbounded.  Hence 
   $$ \mbox{$C_{\dbar,1}(\ov\D)\setminus  C^1(\ov\D)\not=\emp$ and $C^1(\ov\D)\setminus C_{\dbar,1}(D)
   \not=\emp$}.$$

   On the other hand,  $C_{\dbar,1}(K)$  is strictly bigger than $C^2(K)$. 
   The example above also   shows
   that if $f\in C_{\dbar,1}(K)$, then $\ov f$ may not belong to $C_{\dbar,1}(K)$
   (note that $\dbar {\;\ov f} =\ov {\partial f}$).\\
  
{\bf Question}. Is the restriction algebra $C_{\dbar,1}(K)|_{K^\circ}$  a subalgebra of $C^2(K^\circ)$?
Note that we only assume that $\dbar f=  \frac{u_x-v_y}{2} +i\,\frac{u_y+v_x}{2}$ is  continuously differentiable in $K^\circ$. \\

   A useful algebraic property is that  $C_{\dbar,1}(K)$ is inversionally closed; 
  this means that if $f\in C_{\dbar,1}(K)$
  has no zeros on $K$, then $1/f\in C_{\dbar,1}(K)$.
  
  \noindent A generalization to $n$-tuples 
  (=solution to the B\'ezout equation ${\sum_{j=1}^n x_jf_j=1}$)
   is given by the following theorem.

  \begin{theorem}\label{bez}
  Suppose that the functions $f_1,\dots,f_n\in C_{\dbar,1}(K)$ have no common zero on $K$.
  Then the B\'ezout equation $\sum_{j=1}^n x_jf_j=1$ admits a solution $(x_1,\dots,x_n)$
  in $C_{\dbar,1}(K)$. 
  \end{theorem}
  We present two proofs.
  \begin{proof}
  (1) Let $$q_j:=\frac{\ov {f_j}}{\sum_{k=1}^n |f_k|^2}.$$
  Then $q_j\in C(K)$ and $\sum_{j=1}^nq_jf_j=1$. By Weierstrass' approximation theorem
  choose a polynomial $p_j(z,\ov z)$ such that  on $K$
  $$ |p_j-q_j|\leq \bigl(2\sum_{k=1}^n ||f_j||_\infty\bigr)^{-1}.$$
  Then 
  $$\bigl|\sum_{j=1}^n p_jf_j\bigr|\geq\left|\sum_{j=1}^n q_jf_j\right| -
  \left|\sum_{j=1}^n (p_j-q_j)f_j\right|\geq 1-\frac{1}{2}=\frac{1}{2}.$$
  
  Note that $\sum_{j=1}^n p_jf_j\in C_{\dbar,1}(K)$. 
  Because $C_{\dbar,1}(K)$ is inversionally closed, we get that
  $$x_j=\frac{p_j}{\sum_{k=1}^n p_kf_k}\in C_{\dbar,1}(K).$$
  Since  $\sum_{j=1}^n x_jf_j=1$,  we see that $(x_1,\dots, x_n)$ is
   the desired solution to the B\'ezout equation in $C_{\dbar,1}(K)$.\\
  
  (2) This is similar to  \cite[Exercice 25, p. 139]{an}.   Let $f_j$ be 
  continuously extended to an open  neighborhood  $U$ of $K$ so
   that $\sum_{j=1}^n |f_j|\geq \delta>0$ on $U$.  
  Let $E_j=\{z\in U: |f_j(z)|>0\}$. Since the functions $f_j$ have no common zeros
  on $U$,  $\Union_{j=1}^n E_j=U$.  
  Let $\{\alpha_j:j=1,\dots,n\}$ be a
   $C^\infty_c$-partition
  of  unity subordinate to the open covering $\{E_1,\dots, E_n\}$ of $K$; that is, 
   $$\alpha_j\in C^\infty_c(\C), 
   \sp 0\leq \alpha_j\leq 1,\sp \supp \alpha_j\ss E_j, \sp\sum_{j=1}^n \alpha_j=1 \;{\rm on}\; K$$
(see \cite [p. 162]{ru}).
Then $x_j:=\alpha_jf_j^{-1}\in C_{\dbar,1}(K)$ and 
$$\sum_{j=1}^n x_jf_j=\sum_{j=1}^n  \alpha_j =1.
\eqno\qedhere$$
\end{proof}

     Next we deal with the generalized B\'ezout equation $\sum_{j=1}^n x_j f_j=f$.
     
    \begin{proposition}
Let $f,f_j\in  C_{\dbar,1}(K)$ and suppose that $f$ vanishes in a neighborhood 
(within $K$) of $\Inter_{j=1}^n Z(f_j)$. Then $f\in I_{C_{\dbar,1}(K)}(f_1,\dots, f_n)$.
\end{proposition}
\begin{proof}

Let $V\ss\C$ be an open set satisfying $\Inter_{j=1}^n Z(f_j)\ss V\inter K\ss Z(f)$. If
$S=K \setminus V$, then  $\inf_S\sum_{j=1}^n|f_j|\geq \delta>0$. 
As usual we  extend  the functions $f_j$  to continuous functions in $C_c(\C)$ 
and denote these extensions by the same symbol. Let 
$U_j= \{z\in \C: |f_j|>0\}$, $j=1,\dots, n$. Then $S\ss\Union_{j=1}^nU_j$. Let 
$\{\alpha_j:j=1,\dots,n\}$ be a $C^\infty_c$-partition
  of  unity subordinate to the open covering  $\{U_1,\dots, U_n\}$  of $S$.
  Then $x_j:=\alpha_jf_j^{-1}\in C_{\dbar,1}(K)$ and 
$$\sum_{j=1}^n x_jf_j=\sum_{j=1}^n  \alpha_j =1 \text{\;on $S$}.$$
Noticing that $f\equiv 0$ on $K\setminus S$, we get 
$$\sum_{j=1}^n (fx_j)f_j=f  \text{\;on $K$},$$
where  $fx_j\in C_{\dbar,1}(K)$.
 \end{proof}

  \section{A new proof of the Nullstellensatz for $A(K)$}\label{2}

Here we  give yet another elementary proof of  the ``Corona Theorem" (or Nullstellensatz)
for the algebra $A(K)$ (see \cite{mr} for the preceding one). We use Wolff's $\dbar$-method
(see for example  \cite[p. 130 and p. 139]{an}). 
  Our main tool will be the following well-known theorem (see \cite{an}, \cite{na}, and \cite{mr}).
  Here $\sigma$ denotes the 2-dimensional Lebesgue measure.
\begin{theorem}\label{cr}
Let $K\ss\C$ be compact,  $f\in C(K)\inter C^1(K^\circ)$ and let  
$$u(z)=-\frac{1}{\pi}\dint_{K} \frac{f(w)}{w-z}d\sigma(w).$$
Then
\begin{enumerate}
\item $u\in C(\hat \C)$, 
\item $u\in C^1(K^\circ)$ and holomorphic outside $K$,
\item  $\dbar u=f$ on $K^\circ$.
\end{enumerate}

\end{theorem}

\begin{theorem}\label{corona}
Let $K\ss\C$ be a compact set  
and let $f_j\in A(K)$. Then the B\'ezout equation $\sum_{j=1}^n h_jf_j=1$ admits a solution
 $(h_1,\dots, h_n)$ in $A(K)$ if and only if the functions $f_j$ have no common zero on $K$.
\end{theorem}

\begin{proof} Assume that $\sum_{j=1}^n|f_j|\geq \delta>0$ on $K$. Applying Theorem \ref{bez},
 there is a solution $(x_1,\dots, x_n)\in C_{\dbar,1}(K)^n\ss C(K)^n$
to $\sum_{j=1}^n x_jf_j=1$. Now we use  Wolff's method to solve a system of $\dbar$-equations.
Consider  $\bs f=(f_1,\dots, f_n)$  as a row matrix; its transpose is denoted by $\bs f^t$.
Let $|\bs f|^2= \sum_{j=1}^n |f_j|^2$; that is $|\bs f|^2=\ov{\bs f}\bs f^t$.

The B\'ezout equation now reads as $\bs x\bs f^t=1$. It is well-known  
(see for instance \cite[p. 227]{mw})
that any other solution
$\bs u\in C(K)$ to the B\'ezout equation  $\bs u\bs f^t=1$ is given by 
$$\bs u^t= \bs x^t+ H \bs f^t,$$
or equivalently
$$\bs u=\bs x-\bs f H,$$
where $H$ is an $n\times n$ antisymmetric matrix over $C(K)$; that is $H^t=-H$.

Let $$F=\left( \left(\dbar{\bs x}^t\cdot \ov{\bs f}\right)^t- \dbar{\bs x}^t\cdot \ov{\bs f}\right)
\frac{1}{|\bs f|^2}.$$
Since $\bs x\in C_{\dbar,1}(K)^n$, we see that
  $F$ is an antisymmetric  matrix over $C(K)\inter C^1(K^\circ)$. Thus,  by Theorem \ref{cr},
the system $\dbar H= F$ admits  a matrix  solution $H$  over $C(K)\inter C^1(K^\circ)$.
 Note that $H$ can be chosen to be  antisymmetric, too.

 It is now  easy to check that on $K^\circ$, $\dbar \bs u=\bs 0$. In fact
$$
\dbar \bs u=\dbar {\bs x}-  \bs f\cdot \dbar H=\dbar{\bs x}-\bs f\cdot\left( \ov{\bs f}^t \cdot \dbar{\bs x} -\dbar{\bs x}^t \cdot \ov{\bs f}\right)\frac{1}{|\bs f|^2}
$$
$$= \frac{(\bs f \cdot \dbar {\bs x}^t )\; \ov{\bs f}}{|\bs f|^2}=
  \frac{\bigl(\dbar( \bs f \cdot\bs x^t)\bigr) \;
\ov{\bs f}}{|\bs f|^2}=
\frac{\bigl(\dbar( \bs x \cdot\bs f^t)^t\bigr)\, \ov{\bs f}}{|\bs f|^2}=\bs 0
$$

 Thus $u=\bs x-\bs f  H\in A(K)$. Hence $\bs u$ is the solution to the
B\'ezout equation in $A(K)$.
\end{proof}

  \section{The principal ideal case}
  
In this section we consider  the   following division problems:
  let $A$ be one of the algebras $C^m(K), A^m(K)$ and $C_{\dbar,1}(K)$. 
  Determine the best constant $N\in \N$ such that $|f|\leq |g|$
  implies $f^N\in I_A(g)$; that is $g$ divides $f^N$ in $A$.  Note that this is the  ($n=1$)-case  of  Problem \ref{N=3}. 
  
  We need to introduce a special class of compacta.
  
  \subsection{ Locally L-connectedness}
  
  \begin{definition}\label{lcon}
 A compact set $K\ss\C $ satisfying $K=\ov{K^{\circ}}$ is said to be {\it locally  L-connected} if for every $z_0\in \partial K$
 there is an open neighborhood $U$ of $z_0$  in $K$  and a constant $L>0$ such that every point 
 $z\in U\inter K^\circ$ can be joined with $z_0$ by a piecewise $C^1$-path $\gamma_z$ 
 entirely contained
 in $K^\circ$ (except for the end-point $z_0$) and such that  
 $$L(\gamma_z)\leq L\;  |z-z_0| ,$$
 where $L(\gamma_z)$ denotes the length of the path $\gamma_z$.
 \end{definition}

 As examples we mention  closed disks,  compact convex sets with interior, and  finite unions of these 
 sets.  Counterexamples, for instance,
 are  comb domains,  spirals  having infinite length and certain disjoint unions of infinitely many disks
 (see below). 
 \begin{figure}[h]
   \hspace{4cm}
   \scalebox{.40} {\includegraphics{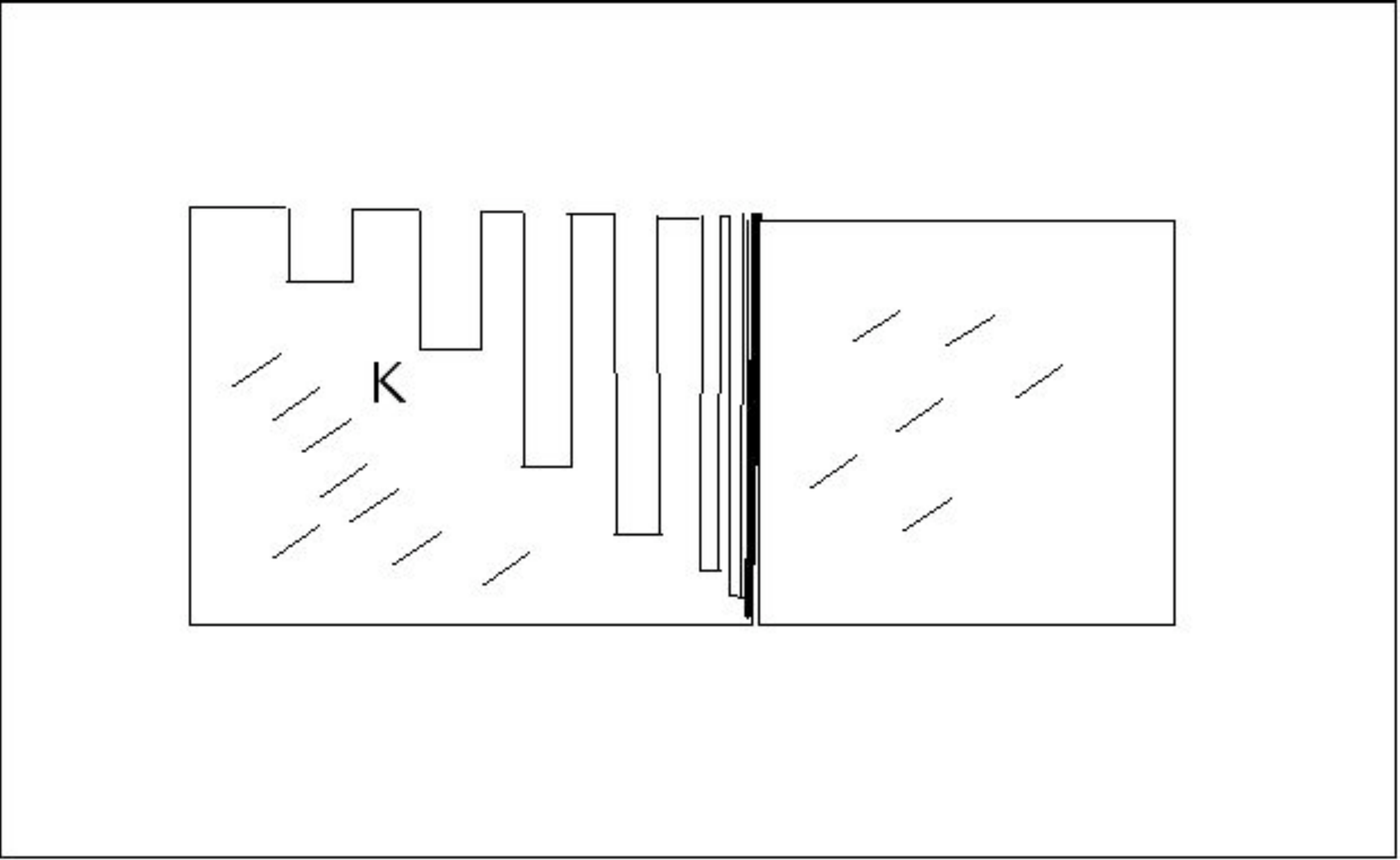}} 
\caption{\label{lconn} The comb domain}
\end{figure}

 \begin{figure}[h]
   \hspace{2cm}
   \scalebox{.30} {\includegraphics{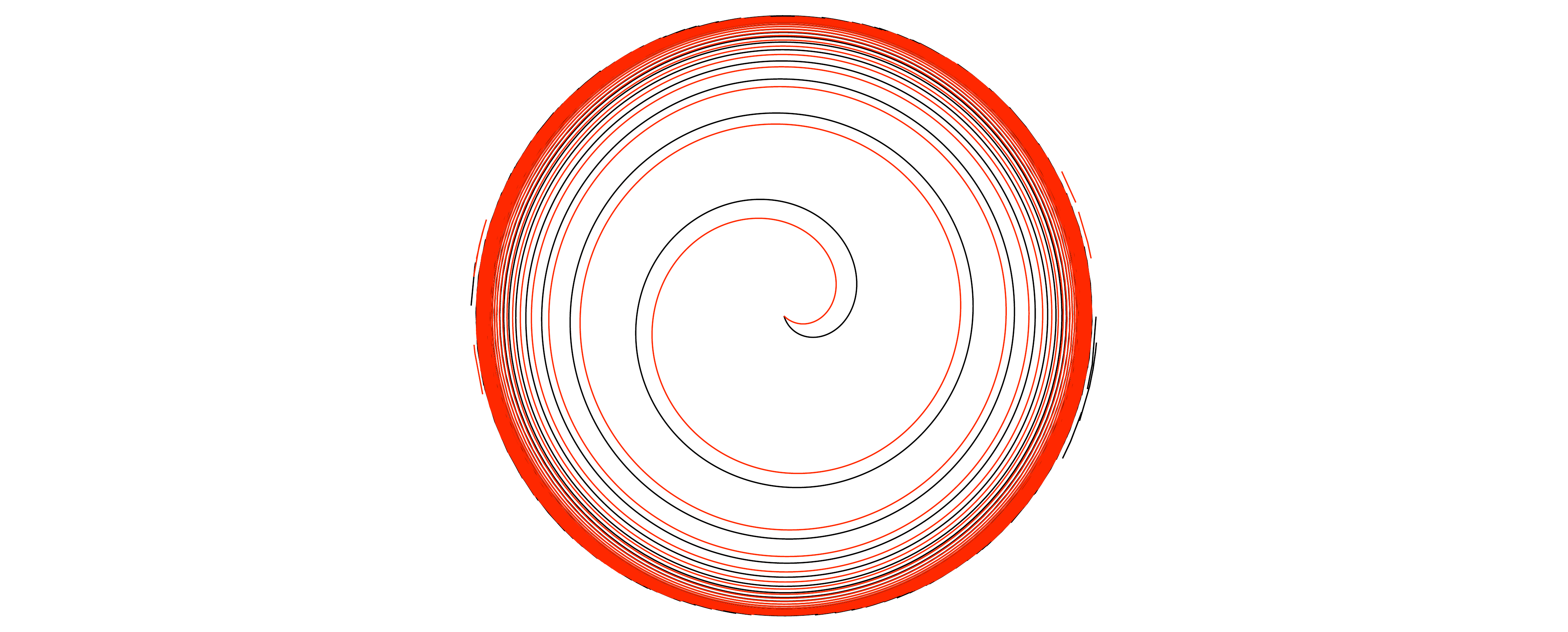}} 
\caption{\label{spir} The spiral domain}
\end{figure}

\begin{figure}[h]
   \hspace{3cm}
   \scalebox{.60} {\includegraphics{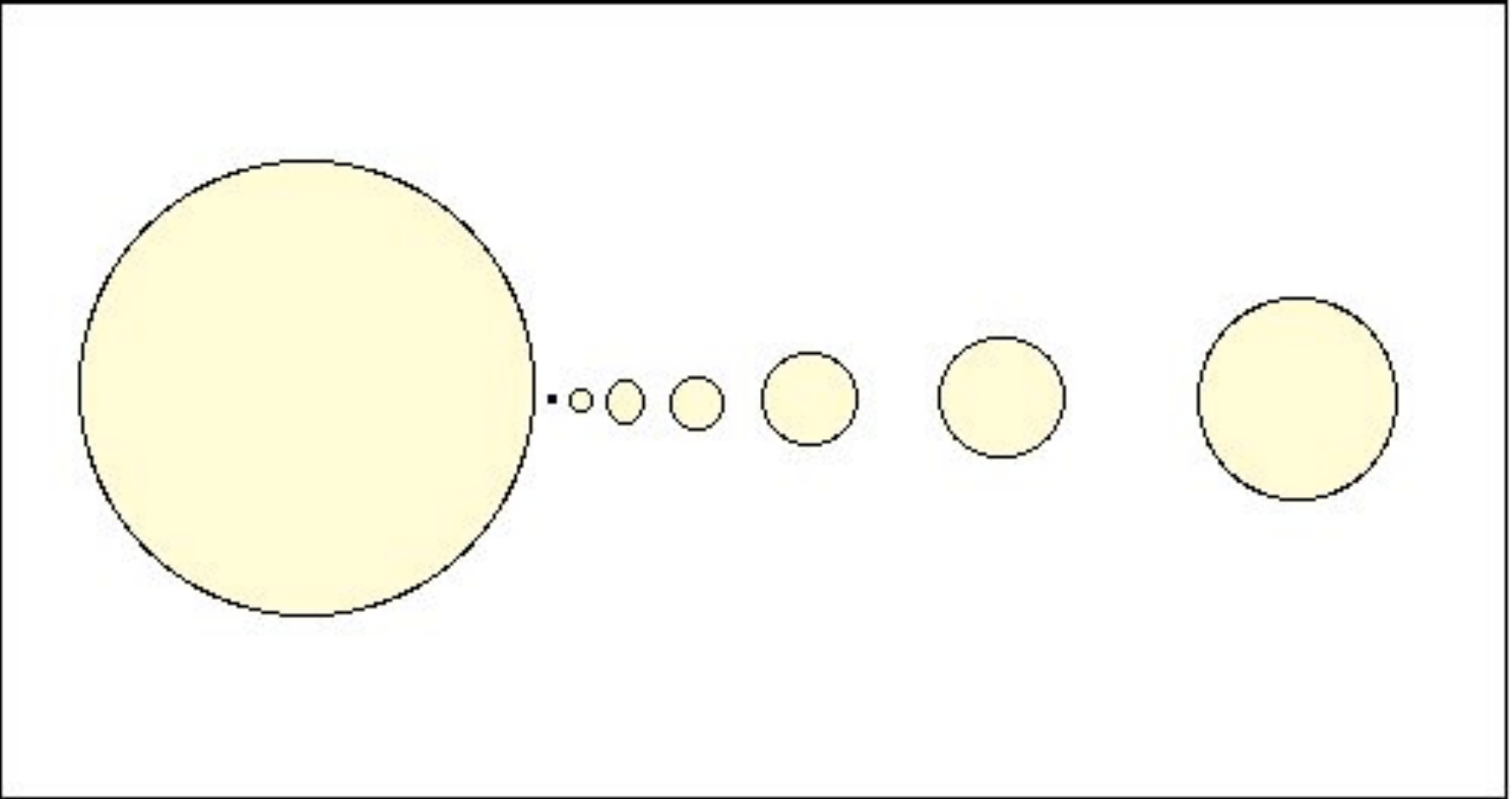}} 
\caption{\label{lcc} A non locally L-connected compactum}
\end{figure}

 

\begin{definition}
A topological space $X$  is said to be {\it locally path-connected} if for every $x\in X$ and  every 
  neighborhood $U\ss X$ of $x$ there exists a neighborhood $V$ of $x$
  such that $V\ss U$ and  for any pair $(u,v)$ of points in $V$ there is a path 
 from $u$ to $v$ lying in $U$. \footnote{ In the ``usual" definition  the path is assumed to lie
 in $V$; these two definitions coincide} 
  \end{definition}

\begin{proposition}
Every locally $L$-connected compactum is locally path-connected.
\end{proposition}
\begin{proof}
Let $z_0\in \partial K=\ov{K^\circ}\setminus K^\circ$. 
Given the  open disk $D(z_0,r)$, we choose the open neighborhood 
 $U$  of $z_0$ in $K$  and $L=L(z_0)>0$ according
to the definition of  local L-connectivity.   In particular, every point 
$z\in D(z_0,r)\inter U\inter K^\circ$
can be joined with $z_0$ by  a curve  $\gamma_{z,z_0}$ contained in $K^\circ\union \{z_0\}$.
Note that  for $z\in U\inter K^\circ$,  $\gamma_{z,z_0}\ss D(z_0, L|z-z_0|)$.  Hence, if we choose $0<r'<r/2L$
so small that $D(z_0,r')\inter K\ss U$, 
then every point $z\in D(z_0,r')\inter K^\circ$ can be joined with $z_0$
by a path $\gamma_{z,z_0}$ contained in 
$$K \inter D(z_0, L|z-z_0|)\ss K\inter D(z_0, r).$$ 
We  still need to show
that every point $z_1\in \partial K\inter  D(z_0, r')$ can be joined with $z_0$ by a path contained in
$K\inter D(z_0,r)$.  By the same argument as above, there is a disk $D(z_1,r'')\ss D(z_0,r')$ 
 such that every point $z\in D(z_1,r'')\inter K^\circ$
can be joined with $z_1$ by a curve $\gamma_{z,z_1}$ contained in 
$$K\inter D(z_1, L(z_1) r'')\ss K\inter D(z_0,  r).$$
Thus the concatenation  of the inverse path $\gamma^{-1}_{z,z_1}$  with $\gamma_{z,z_0}$
joins $z_1$ with $z_0$ within $D(z_0,r)\inter K$.
Hence every point $w\in D(z_0,r')\inter K$ can be joined to $z_0$ by a path contained in $D(z_0,r)\inter K$. If $z_0\in K^\circ$, then the assertion is trivial. 

We conclude that for $z_0\in K$ and $D(z_0,r)$ we find a neighborhood $V$ of $z_0$ in $K$
such that any two points $u,v\in V$ can be joined by a path (passing through $z_0$) 
that stays in $D(z_0,r)$.
Hence $K$ is locally path-connected.
\end{proof}

 We mention that there exist locally connected, path-connected  continua $K$ with $\ov{K^\circ}=K$ that are
 not locally L-connected. For example, let $K$ be the inner  spiral 
 $(\theta+1)^{-1}\leq r(\theta)\leq \theta^{-1}$, 
 $\pi\leq \theta \leq \infty$. See the following figure. 
 Here the point $0\in\partial K$,
 but cannot be joined with any other point $z\in K$ by a rectifiable path.
 We can also take a spiral consisting of ``thick" half-circles with radii $1/n$, \;$n=1,2,\dots$.

 \vspace{5cm}
 \begin{minipage}[b]{6cm}  
 \begin{picture}(6,6) 

 \hspace {1,5cm}{\scalebox{.30}{\includegraphics{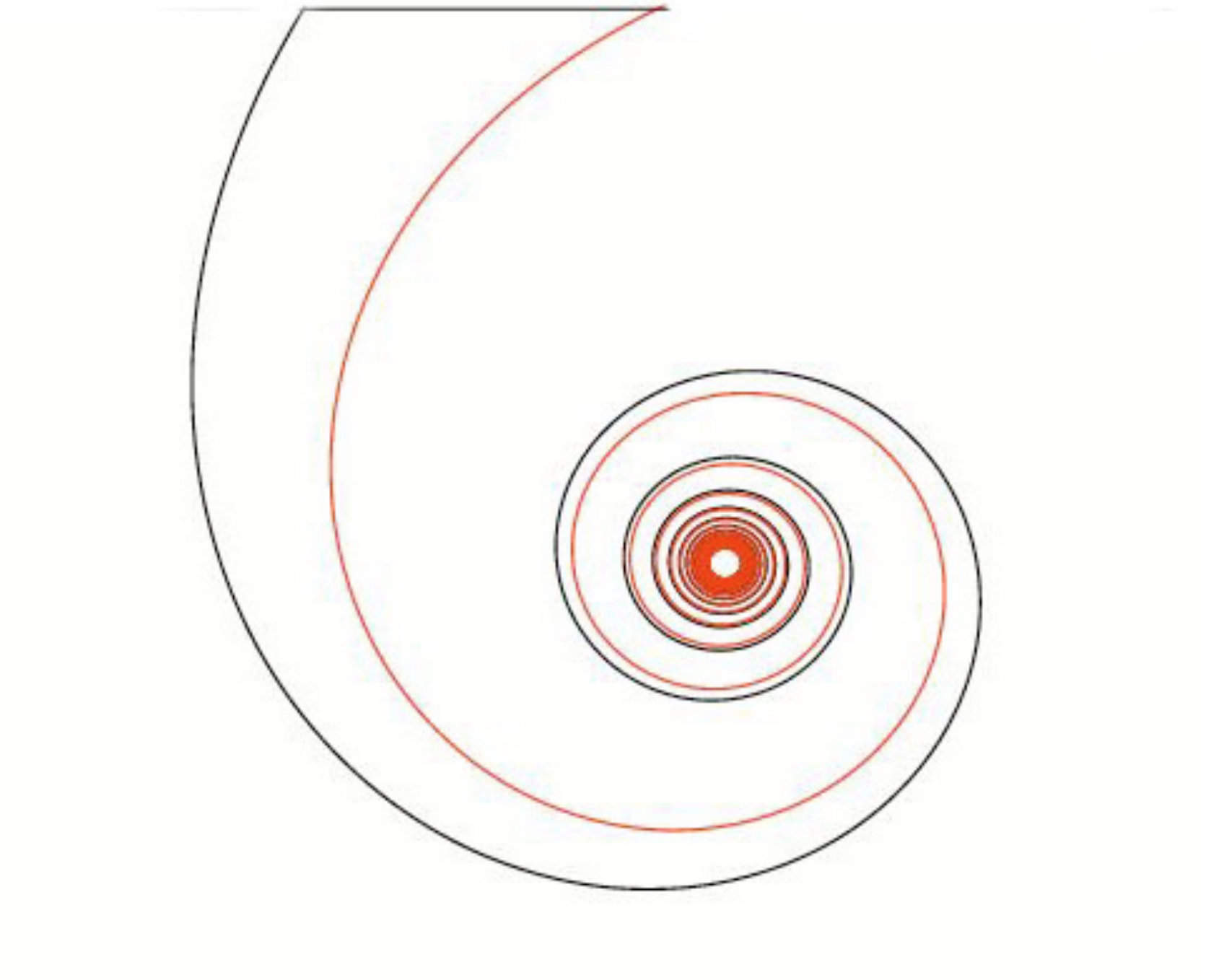}}}

\end{picture}
\end{minipage}
\hfill
\begin{minipage}[b]{9cm}
\begin{picture}(6,6)
\hspace{1cm}
{\scalebox{.40}{\includegraphics{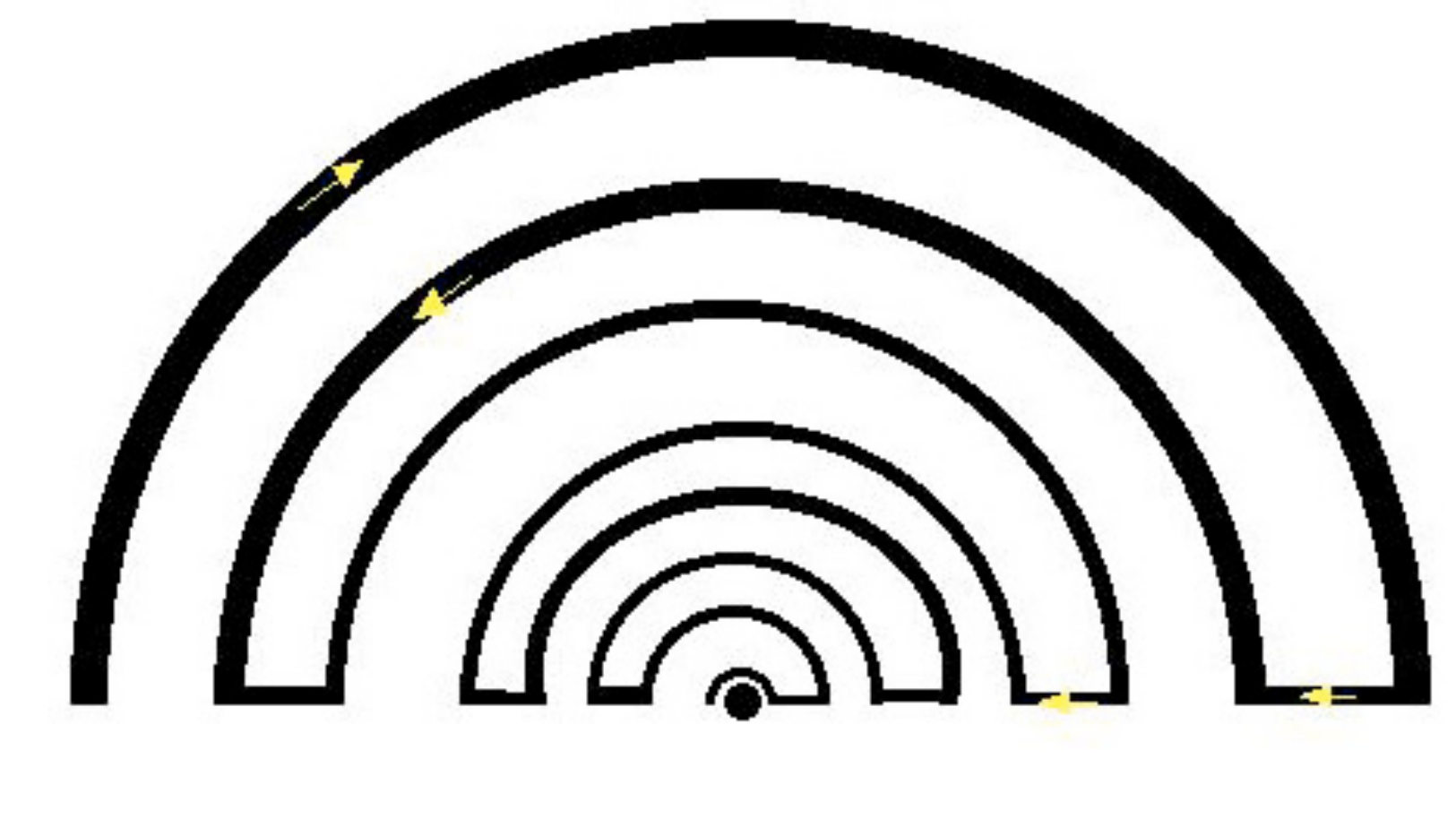}}}
\end{picture}

\end{minipage}

 \subsection{The Taylor formula on the boundary}
  
  We need the following Taylor formula for functions in $A^m(K)$.
  \begin{proposition}\label{taylor}
 (1)  Let $K$ be a  locally $L$-connected compactum. Given $f\in A^m(K)$ and $z_0\in \ov {K^\circ}\setminus K^\circ=\partial K$,  we denote the continuous extension
  of $f^{(j)}$ to $z_0$  by the symbol $f^{(j)}(z_0)$, $(j=0,1,\dots,m)$. Let 
  $$P_m(z,z_0)=f(z_0)+ \frac{1}{1!} f '(z_0)\, (z-z_0)+\dots+ \frac{1}{m!} f^{(m)}(z_0)\, 
  (z-z_0)^m$$
  be the $m$-th Taylor polynomial of $f$ at $z_0$. Then
  $$f(z)=P_m(z,z_0)+ R_m(z),$$
  where  $R_m\in A^m(K)$, $R_m^{(j)}(z_0)=0$ for $j=0,1,\dots, m$ and
  $$\mbox{$R_m^{(j)}(z)=\oh(z-z_0)^{m-j}$  as $z\to z_0$}.$$
  
  Moreover, 
  \begin{equation}\label{abl}
\lim_{z\to z_0\atop  z\in K^\circ} f '(z)=
   \lim_{z\to z_0 \atop z\in K\setminus\{z_0\}} \frac{f(z)-f(z_0)}{z-z_0}.
   \end{equation}
   
   (2) If $K$ is not locally L-connected, then  the equality \zit{abl} does not hold in general.
  \end{proposition}

  \begin{proof}
  
  (1) We first show equality \zit{abl}. Set $\ell:=\dis\lim_{z\to z_0\atop  z\in K^\circ} f '(z)$.
  Let $U\ss K$ be  the  neighborhood of $z_0$  and $L>0$ the associated constant given by 
   the definition \ref{lcon} of locally L-connectedness.
  
  Given $\e>0$,  choose $\delta>0$ so small that $|f '(\xi)-\ell|<\e/L$ for $|\xi-z_0|<\delta$,
  $\xi\in K^\circ$. 
   For $z\in D(z_0,\delta/L)\inter U\inter K^\circ$,  let $\gamma_{z,z_0}$ be a path in 
 $K^\circ\union \{z_0\}$ joining $z$ with $z_0$ and with length $L(\gamma_{z,z_0})\leq L\,|z-z_0|$.
  Note that  $$\gamma_{z,z_0}\ss D(z_0, L\,|z-z_0|)\ss D(z_0,\delta).$$
  
  Then, by integrating along $\gamma_{z,z_0}$  from $z$ to a point $\tilde z_0$
  close to $z_0$ and passing to the limit,  we obtain
  $$f(z)-f(z_0)=\int_{\gamma_{z,z_0}} f '(\xi) \; d\xi.$$
  Hence
  \begin{eqnarray*}
\left|\frac{f(z)-f(z_0)}{z-z_0}-\ell\right| &\leq& \frac{1}{|z-z_0|}\int_{\gamma_{z,z_0}}
 |f '(\xi)-\ell|\; |d\xi| \\
&\leq &\frac{1}{|z-z_0|}\;\frac{\e}{L}  \; L(\gamma_{z,z_0})\leq \e.
\end{eqnarray*}
This confirms \zit{abl}. \\

Since $f\in A^m(K)$, we obviously have $R_m\in A^m(K)$. Moreover, by taking derivatives in $K^\circ$ and extending them to the boundary,  we see that 
$R_m^{(j)}(z_0)=(f-P_m)^{(j)}(z_0)=0$ for $j=0,1,\dots,m$.  Moreover, 
$R_m^{(m)}(z)=\oh(1)$ for $z\to z_0$. As in the proof of \zit{abl},  for $z\in K^\circ\inter U$,
$$R_m^{(m-1)}(z)= \int_{\gamma_{z,z_0}} R_m^{(m)}(\xi) \;d\xi.$$
Hence, 
$$|R_m^{(m-1)}(z)|\leq L |z-z_0|\; \max_{\xi\in \gamma_{z,z_0}} | R_m^{(m)}(\xi)|=\oh(z-z_0).$$
Also,
 \begin{eqnarray*}
 R_m^{(m-2)}(z)&=& \int_{\gamma_{z,z_0}} R_m^{(m-1)}(\xi) \;d\xi\\
 &=&  \int_{\gamma_{z,z_0}}(\xi-z_0)\;  \frac{R_m^{(m-1)}(\xi)}{\xi-z_0} \;d\xi,
 \end{eqnarray*}
and so
\begin{eqnarray*}
|R_m^{(m-2)}(z)|  &\leq &  \max_{\xi\in \gamma_{z,z_0}}  
\frac{|R_m^{(m-1)}(\xi)|}{|\xi-z_0|} \int_{\gamma_{z,z_0}}|\xi-z_0|\;|d\xi|\\
&\leq &  \oh(1) L(\gamma_{z,z_0})^2=\oh(z-z_0)^2.
 \end{eqnarray*}

Using backward induction, we obtain the assertions  that for $j=m,m-1,\dots, 1,0$
$$R_m^{(j)}(z)=\oh(z-z_0)^{m-j}$$  as $z\to z_0$.\\

  (2) Let $K$ be the union of the disk $D_0=\{z: |z+1|\leq 1\}$ and the disks 
  $$D_n=\{z: |z-1/n|\leq 1/n^3\}, n=3,4,\dots.$$
  Note that $D_n\inter D_m=\emp$
  for $n\not=m$, $n,m\geq 3$, and that $K=\ov{K^\circ}$, but that $K$ is not locally
  path-connected (at the origin).
  Let $f$ be defined as follows: $f=0$ on $D_0$ and $f=1/\sqrt n$  on
  $D_n$. Then $f$ is continuous on $K$, holomorphic on $K^\circ$ and $f'\equiv 0$ on $K^\circ$.
  It is obvious that $f'$ admits a continuous extension to $K$, namely by the constant $0$.
  However, 
  $$\frac{f(1/n)-f(0)}{1/n-0} = \sqrt n\to \infty.$$
  Thus the continuous extension of $f'$ to $0$  is distinct from the  limit of the associated
  differential quotient: 
  $$0\not= \lim_{z\to 0, z\in K^\circ\atop {\rm Re}\; z>0} \frac{f(z)-f(0)}{z-0}=\infty.$$
  \end{proof}
  
  \begin{corollary}\label{lingi}
Let $f\in A^1(K)$, where $K$ is a locally L-connected compactum.
 Then for every $z_0\in\ov {K^\circ}\setminus K^\circ$ there exists 
$F\in A(K)$ with $F(z_0)=f '(z_0)$,   respectively
$H\in A^1(K)$ satisfying  $H(z_0)=H'(z_0)=0$ and  $H(z)=\oh(z-z_0)$,  such that 
\begin{eqnarray}
f(z)&= &f(z_0)+(z-z_0)f '(z_0)+ H(z)\\
&=&f(z_0)+(z-z_0)F(z)
\end{eqnarray}
\end{corollary}

{\bf Remark} The function $F$ in the representation  $f(z)=f(z_0)+(z-z_0)F(z)$
 for $f\in A^1(K)$ may not belong to $A^1(K)$ itself. In fact, let ${f(z)=(1-z)^3 S(z)}$,
$z\in \D$,  where
$$S(z)=\exp\left(-\frac{1+z}{1-z}\right)$$
 is the atomic inner function.
 Then $f\in A^1(\ov\D)$ and $f(z)=(1-z)F(z)$ with $F(z)=(1-z)^2 S(z)$. But 
$F\notin A^1(\ov\D)$. \medskip

\subsection{Division in $A^m(K)$}

We first  consider the cases $m=0,1$.

\begin{theorem}\hfill\label{powers}

\begin{enumerate}
\item [a)] If $f,g\in C(K)$ satisfy  $|f|\leq |g|$,
then $f^2\in  I_{C(K)}(g)$.
\item[b)] If $f,g\in C^1(K)$ satisfy $|f|\leq |g|$ and if $Z(g)$ has no cluster points in $K^\circ$,
then $f^3\in  I_{C^1(K)}(g)$.
 \item[c)]   If $f,g\in A(K)$  satisfy $|f|\leq |g|$, then
 $f^2\in  I_{A(K)}(g)$.
 \item[d)]   If $f,g\in A^1(K)$  satisfy $|f|\leq |g|$, then
 $f^3\in  I_{A^1(K)}(g)$.
 \item[e)]   If $f,g\in A^1(K)$ satisfy  $|f|\leq |g|$, then
 $f^2\in  I_{A^1(K)}(g)$  whenever $K$ is locally L-connected.
 \item [f)]  If $f,g\in C_{\dbar,1}(K)$ satisfy $|f|\leq |g|$, and if $Z(g)$ has no cluster points in $K^\circ$, then $f^4\in  I_{C_{\dbar,1}(K)}(g)$.
\end{enumerate}
The powers  2,3,2,3,2,4 (within $\N$) are optimal. 
\end{theorem}

\begin{proof}

a) Since the quotient $f/g$ is bounded on $K\setminus Z(g)$,
we just have to put $h=f^2/g$ on $K\setminus Z(g)$ and $h=0$ on $Z(g)$
in order to see that $f^2=gh$, where $h\in C(K)$. In fact, if $z\in\partial Z(g)$,
then $g(z)=0$ and so $|f|\leq |g|$ implies $f(z)=0$, too. Hence, for any sequence $z_n$
in $K\setminus Z(g)$ with $z_n\to z_0$ we see that $f^2/g(z_n)\to 0$. 
Thus $h\in C(K)$ and so $f^2\in I_{C(K)}(g)$. 

\centerline{\rule{6cm}{0,2mm}}\medskip

b) Let $h=f^3/g$   on $K\setminus Z(g)$ and $h=0$ on $Z(g)$.
Then by a),  $h\in C(K)$ and $h\in C^1(K\setminus Z(g))$.
Let $D$ be one of the derivatives $d/dx$ or $d/dy$. 
Then
  $$\mbox{$\dis D h= \frac{3gf^2(D f)-f^3(D g)}{g^2}$  on 
$K^\circ\setminus Z(g)$}.$$

Because $f,g\in C^1(K)$,  $Df|_{K^\circ}$ and $Dg|_{K^\circ}$ extend continuously
 to $ \ov{K^\circ}$. So
$Dh|_{K^\circ\setminus Z(g)}$ extends continuously to $\ov{K^\circ}\setminus Z(g)$.
Moreover, the assumption  $|f|\leq |g|$ implies that   $|Dh|\leq \kappa |g|$ on $K^\circ\setminus Z(g)$.
Hence, $Dh|_{K^\circ\setminus Z(g)}$ extends continuously to $\ov{K^\circ}\inter Z(g)$
(with value 0).
By combining both facts, we conclude that $Dh|_{K^\circ\setminus Z(g)}$ extends
continuously to $\ov {K^\circ}$.  Tietze's theorem now yields the  extension
of $Dh|_{\ov {K^\circ}}$ to a continuous function on $K$.

 If $a\in Z(g)\inter K^\circ$ is an isolated point
 then, by the mean value theorem, $D h (a)$ exists and coincides with  the continuous extension of 
 $Dh|_{K\setminus Z(g)}$  at $a$. Thus $h\in C^1(K^\circ)$.   
 
Putting it  all together, we  conclude that $h\in C^1(K)$.
 Hence $f^3=gh\in I_{C^1(K)}(g)$. \\

  We remark that if $Z(g)\inter K^\circ$ is not discrete, then we are not always able to conclude
  that $h$ is differentiable at points  in $\ov{K^\circ\setminus Z(g)}\inter Z(g)\inter K^\circ$.
\footnote{Note that
 the possibility of a continuous  extension of the partial derivatives to $E=Z(g)$  does not mean
 that the function itself is differentiable at the points in $E$. See Proposition \ref{taylor}(2) or
 just consider the Cantor function $C$
 associated with the Cantor set $E$ in $[0,1]$. Here the derivative of $C$ is zero at every point
 in $\R\setminus E$, $C(0)=0$, $C(1)=1$, $C$ increasing, but  $C$ itself does not belong to $C^1(\R)$.}

\centerline{\rule{6cm}{0,2mm}}\medskip
 
 c) If $f,g\in A(K)$, then $|f|\leq |g|$ implies that every zero of $g$ is a zero of $f$.
By Riemann's singularity theorem, the boundedness of the quotient  $f/g$ around isolated
zeros of $g$ within $K^\circ$  implies that $f/g$, and hence $f^2/g$, are holomorphic.
If $g$ is constantly zero on a  component $\Omega_0$ of $K^\circ$, 
then $f\equiv 0$ on $\Omega_0$, too.   So we have just to define $f^2/g=0$ there.
Using a) and the fact that $ \partial (Z(g)^\circ)\ss \partial K$, 
 we conclude that
$$h(z)=\begin{cases}   \frac{f^2(z)}{g(z)} & \text{if $z\in K\setminus Z(g)$}\\
                                       ~~~0 &\text{if $z\in Z(g)$}
\end{cases}
$$
 belongs to $A(K)$. Hence $f^2\in I_{A(K)}(g)$. 
 
 \centerline{\rule{6cm}{0,2mm}}\medskip

d) 
Let $S=K\setminus Z(g)^\circ$. Since the zeros of the holomorphic function $g$ do not accumulate
at any points in $S^\circ$,    we deduce from b) that  on $S$, $f^3=k g$ for some $k\in C^1(S)$. 
Moreover, $k=0$ and $Dk=0$ on $Z(g)\inter S$.

 If we let $h=k$ on $S$ and $h=0$ on $Z(g)$,
then $h\in C(K)$.  Let 
$$K^\circ=(\Union_n \Omega_n)\union (\Union_j \Omega_j')$$ 
where the $\Omega_n$ are those components of $K^\circ$  containing only 
 isolated zeros of $g$ and where
the $\Omega_j'$ are those components of $K^\circ$  where $g$ vanishes identically.
Note that  $\Omega_n$ and $\Omega_j'$ are open sets.   
Since the quotient $f^3/g$ extends holomorphically at every isolated zero
of $g$, we conclude that $h$ is holomorphic on each of these components. 
Hence $h\in A(K)$. Moreover, 
$$\footnote{Actually we have $\partial(Z(g)^\circ)\ss\partial S=\partial K\sp$\sp}\sp
\partial K=\partial S \union  \partial(\Union_j  \Omega_j').$$
Since $k\in C^1(S)$ and $Dh=0$ on $\Union_j  \Omega_j'$, we deduce that
$Dh$ has a continuous extension to $\partial K$. Consequently,  $h\in A^1(K)$.

\centerline{\rule{6cm}{0,2mm}}\medskip

e) Let $f,g\in A^1(K)$. Define $h$ by
$$h(z)=\begin{cases} \frac{f^2(z)}{g(z)} & \text{if $z\in K\setminus Z(g)$}\\
                                     0 & \text{if $z\in Z(g)$.}
\end{cases}
$$

 Then, by c), $h\in A(K)$. Since $K$ is locally L-connected, 
and so $K=\ov{K^\circ}$, we may apply Corollary \ref{lingi}. Hence,
 for every $z_0\in\partial K$ there
exists $F,G\in  A(K)$ 
such that $f(z)=f(z_0)+(z-z_0)F(z)$ and $g(z)=g(z_0)+{(z-z_0)G(z)}$. In particular, if $g(z_0)=0$,
$|f|\leq |g|$ implies that  $f(z_0)=0$ and  $|F|\leq |G|$. If additionally $g'(z_0)=0$, then 
$G(z_0)=F(z_0)=0$; hence $f '(z_0)=0$, too.  
 Now, for $z\in K^\circ\setminus Z(g)$,
\begin{eqnarray}\label{deri}
\partial\left(\frac{f^2(z)}{g(z)}\right)&= &\frac{2f(z)g(z) f '(z)-f^2(z)g'(z)}{g^2(z)}\\
&=&2\frac{f(z)}{g(z)} f '(z) - \left(\frac{f(z)}{g(z)}\right)^2 g'(z).
\end{eqnarray}

{\bf Case 1}
Let $z_0\in \ov{K\setminus Z(g)}\inter Z(g)$. This means  that  $z_0$ is in the
boundary of $Z(g)$ with respect to the topological space $K$. Note also that by 
our global assumption,
$z_0\in\partial K$.
 
{\bf 1.1} If $g'(z_0)=f '(z_0)=0$, then, by (\ref{deri}), the boundedness of the quotient 
$(f/g)|_{\ov{K^\circ}\setminus Z(g)}$ implies 
that $\left(\partial\bigl(\frac{f^2(z)}{g(z)}\bigr)\right)\bigr|_{\ov{K^\circ}\setminus Z(g)}$ admits a continuous extension to $z_0$.

{\bf 1.2} If $g'(z_0)\not=0$, then $G(z_0)\not=0$. Hence 
$$\frac{f(z)}{g(z)}=\frac{(z-z_0) F(z)}{(z-z_0) G(z)}=\frac {F(z)}{G(z)},$$
is continuous at $z_0$. Thus, by (\ref{deri}), the continuity of
$f'$ and $g'$ implies that $\left(\partial\bigl(\frac{f^2(z)}{g(z)}\bigr)\right)\bigr|_{\ov{K^\circ}\setminus Z(g)}$ admits a
 continuous extension to $z_0$.
 
 {\bf Case 2}  Let $z_0\in {\rm int}\,Z(g)\inter \partial K$, the interior being taken in the topological space  $K$.
 Then there is an open set $U$ in $\C$ containing $z_0$ such  that  $U\inter K\ss Z(g)$.
 Since $K=\ov{K^\circ}$, 
  $V:=U\inter K^\circ\not=\emp$. Moreover 
$g\equiv 0$ on $V$. 
 Note that by the definition of $h$,  $(f^2/g)|_{\ov{K^\circ}\setminus Z(g)}$ 
 has a continuous extension 
 to $K$ with values 0 on $Z(g)$.  Thus the derivative of $h$ is zero on $V$, and 
 the derivative of the null-function $h|_V$
 has  a continuous  extension to $z_0$ with value 0.  
 
 Combining both cases, we conclude that $h\in A^1(K)$ and so $f^2\in I_{A^1(K)}(g)$.
 
\centerline{\rule{6cm}{0,2mm}}\medskip

f) Let $h=f^4/g$   on $K\setminus Z(g)$ and $h=0$ on $Z(g)$. Using b), we  know
that $h=f (f^3/g)|_{K^\circ\setminus Z(g)}$ admits an extension to a function in $C^1(K^\circ)$.
(Here we have used that the zeros of $g$ are isolated in $K^\circ$).
Next we have to consider the second-order derivatives. Since on $K^\circ\setminus Z(g)$
\begin{equation}\label{eine}
\dbar h=  \frac{4 g \,(\dbar f)f^3-f^4 \,\dbar g}{g^2}
 \end{equation}
 we obtain
{\tiny $$
 \partial(\dbar h)=\frac{g^2\bigl[ 4(\partial g) (\dbar f) f^3+ 4g (\partial\dbar f) f^3+ 12g (\dbar f)
 f^2(\partial f)  -4f^3(\partial f)(\dbar g) -f^4 (\partial \dbar g)\bigr] - 8g^2 (\dbar f) f^3 (\partial g)+
 2f^4 g (\dbar g)(\partial g) }{g^4}
 $$
 }
 and {\tiny $$
  \dbar(\dbar h)=\frac{g^2\bigl[ 4(\dbar g) (\dbar f) f^3+ 4g (\dbar\, \dbar f) f^3+ 12g (\dbar f)^2
 f^2  -4f^3(\dbar f)(\dbar g) -f^4 (\dbar \,\dbar g)\bigr] - 8g^2 (\dbar f) f^3 (\dbar  g)+
 2f^4 g (\dbar g)^2}{g^4}.
 $$
 }
 
 Using that $|f|\leq |g|$ we conclude from \zit{eine} that  on $K^\circ\setminus Z(g)$,   
 $$|\dbar h|\leq   4|\dbar f| \; |g|^2 + |\dbar g| \; |g|^2\leq C |g|^2.$$
 Hence $\dbar h|_{K^\circ\setminus Z(g)}$ admits a continuous extension to 
 $\ov{K^\circ}$ with value 0 on $Z(g)\inter \ov{K^\circ}$. Using
 Tietze's theorem we get a continuous  extension to $K$, too. 
 
 To show that $\dbar h$ is in $C^1(K^\circ)$,  we use the formulas above to conclude that
for every $z_0\in \ov{K^\circ\setminus Z(g)}\inter K^\circ$ there is  a small neighborhood 
$U$ of $z_0$ such that on $U\setminus Z(g)$
 $$\max\{|\dbar(\dbar h)|, \; |\partial(\dbar h)|\} \leq C\; |g|.$$
 Thus $\dbar(\dbar h)|_{K^\circ\setminus Z(g)}$ and $\partial(\dbar h)|_{K^\circ\setminus Z(g)}$
 admit continuous extensions to $K^\circ$ with value 0.
 Since the zeros of $g$ within $K^\circ$  are isolated,  we deduce that 
 $\dbar h|_{K^\circ\setminus Z(g)}$ belongs to $C^1(K^\circ)$.
  
  Thus $h\in C_{\dbar,1}(K)$ and so $f^4\in I_{C_{\dbar,1}(K)}(g)$.
 
  \hrulefill \medskip
 
 To show the optimality of the powers, let $K$ be the closed unit disk $\ov\D$ and consider for
 a) and c) the functions $f(z)=(1-z)S(z)$ and $g(z)=1-z$ where 
 $$S(z)=\exp\left(-\frac{1+z}{1-z}\right)$$ is the atomic inner function. 
 Then on $\D$, $f/g=S$ and $|f|\leq |g|$. 
 But $S$  does not admit a continuous extension to $\ov\D$.
 
 \hrulefill \medskip
   
 For b), let $f(z)=z$ and $g(z)=\ov z$. Then the $\dbar$-derivative of
$f^2/g$ outside $0$ does not admit  a continuous extension to $0$; in fact, 
$\dbar (f^2(z)/g(z))=\dbar(z^2/\ov z)=-z^2/\ov z^2$,
a function that is discontinuous at $0$.

 \hrulefill \medskip
 
For e) let $f(z)=(1-z)^3S(z)$ and $g(z)=(1-z)^3$. Then $f,g\in A^1(\ov\D)$,
$|f|\leq |g|$ but $f/g\notin A^1(\ov\D)$.
 
  \hrulefill \medskip
 
For f)  Let $f(z)=z$ and $g(z)=\ov z$. Then the $\partial$-derivative
of $\dbar(f^3/g)$ outside $0$ does not admit  a continuous extension to $0$; in fact,
$$\dbar (f^3(z)/g(z))=\dbar( z^3/\ov z)= -z^3/\ov z^2.$$
Hence 
$$\partial (\dbar (f^3(z)/g(z)))=-\partial (z^3/\ov z^2)=-3 z^2/\ov z^2.$$

  \hrulefill \medskip
 
 For d) we necessarily have to consider a  compactum that is not locally L-connected. 
 Moreover, as the proof of e) shows, the extensions  of $f'|_{K^\circ\setminus Z(g)}$ and 
$g'|_{K^\circ\setminus Z(g)}$ to a point $z_0\in \partial K\inter Z(g)$ cannot both be zero.
 As an example we take the following compact set $K$ (see figure \ref{sec}):
 
 $$K= \{0\} \union \Union_{n=1}^\infty \left\{z\in \C: \frac{1}{2^{2n+1}}\leq |z|  \leq\frac{1}{2^{2n}}: |\arg z|\leq \pi/4\right\}.$$
 
  \begin{figure}[h]
   \hspace{4cm}
   \scalebox{.40} {\includegraphics{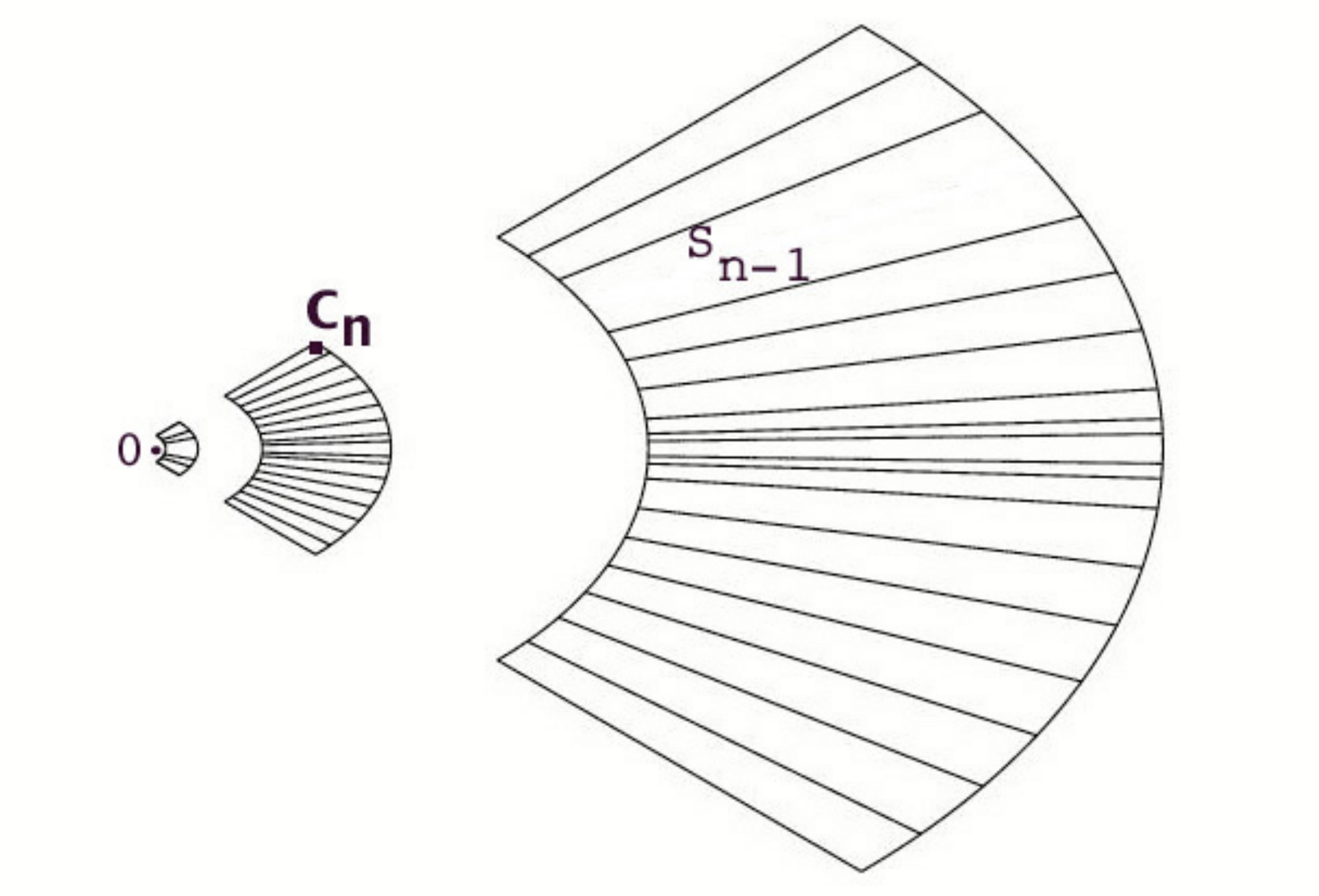}} 
\caption{\label{sec} The sectors}
\end{figure}

Let  $C_n$  be the upper right corner of the sector 
$S_n:=\left\{z\in \C: \frac{1}{2^{2n+1}}\leq |z|  \leq\frac{1}{2^{2n}}: |\arg z|\leq \pi/4\right\}$.
Define the functions $f$ and $g$ by  $f(z)=z$ and 
 $$g(z)= \begin{cases}  \ov C_n &\text{if $z\in S_n$}\\
 0 & \text{if $z=0$}.
 \end{cases}
 $$
 Then $f$ and $g$ belong to  $A^1(K)$ and $|f|\leq |g|$; note that for  $z\in S_n$
 $$|f(z)|\leq \max_{u\in S_n} |u| = 2^{-2n}=|C_n|=|g(z)|.$$
  Since  $g'\equiv 0$ on $K^\circ$, we obtain   
 for  every $z\in K$,  $z\not=0$: 
 $$\Delta(z):=\frac{d}{dz}\left(\frac{f^2(z)}{g(z)}\right)= 2 f '(z) f(z)\frac{1}{g(z)} = 2\frac{z}{g(z)}.$$
 If $z_n=C_n = r_n e^{i\pi/4}$, then  $\Delta(z_n)= 2 e^{2i\pi/4}$; but if
 $z_n=\ov C_n$, then $z_n\in K$ and $\Delta(z_n)=2$.
 Thus $\dis\lim_{z\to 0\atop z\in K\setminus\{0\}} \Delta(z)$ does not exist.
 Hence $f^2/g\notin A^1(K)$.  
\end{proof}

\bigskip

In order to study division in  the algebras $C^m(K)$ and $A^m(K)$, $m\geq 3$,
we use a variant of the Fa\`a di Bruno formula, given in  \cite{moelem}. 

\begin{theorem}\label{faa}
   Let  $$\M^j=\{\mbox{\boldmath$ k$}=(k_1,\dots, k_j)\in (\N^*)^j, \;k_1\geq k_2\geq \cdots\geq k_j\geq 1\}$$
 be the set  of ordered multi-indices in $\N^*=\{1,2,\dots\}$. Then for every $f,g\in C^n(\R)$
  \begin{equation}\label{mo}
   (f\circ g)^{(n)}(x)= \sum_{j=1}^n f^{(j)}(g(x))\biggl(
 \sum_{\kvn} C_{\mbox{\footnotesize \boldmath$k$}}^n\,
g^{(\mbox{\footnotesize \boldmath$k$})}(x)\biggr),
\end{equation}
where
 $$C_{\mbox{\boldmath$ k$}}^n= \frac{\dis{n\choose \mbox{\boldmath$ k$}} }
 {\dis\prod_i n(\mbox{\boldmath$ k$}, i)!}.$$
Here  $g^{(\mbox{\footnotesize \boldmath$k$})}=g^{(k_1)}g^{(k_2)}\cdots g^{(k_j)}$,
${n\choose \mbox{\boldmath$ k$}}$ is  the multinomial coefficient   defined by
 $\dis {n \choose \mbox{\boldmath$ k$}}= \frac{n!}{k_1! k_2!\dots k_j!}$, where
 $|\kv|:=k_1+\dots +k_j=n$, 
and  $n(\mbox{\boldmath$ k$}, i)$ is the number of
times the integer $i$ appears in the  $j$-tuple $\mbox{\boldmath$ k$}$ ($i\in \N^*$ and $\mbox{\boldmath$ k$}\in (\N^*)^j$).
\end{theorem}

\begin{theorem}\label{divi}
Let $f,g\in A^m(K)$ and suppose that $|f|\leq |g|$.
Then, for every $n\in\N$ with  $0\leq n\leq m$, the following estimate
 holds on $K^\circ\setminus Z(g)$: 
$$\left(\frac{f^{m+2}}{g}\right)^{(n)}\leq C\, |g|^{m+1-n}.$$
\end{theorem}
\begin{proof}
Let $I(z):=1/z$ and $h(z):=1/g(z)$. Then $h=I\circ g$.   Also,
$$I^{(j)}(z)= j!(-1)^j \frac{1}{z^{j+1}.}$$
By Theorem \ref{faa}, for $1\leq \mu\leq m$, 
$$ h^{(\mu)}=(I\circ g)^{(\mu)}=\sum_{j=1}^\mu (I^{(j)}\circ g)\; \biggl(
\sum_{\kvm} C_{\mbox{\footnotesize \boldmath$k$}}^\mu\,
g^{(\mbox{\footnotesize \boldmath$k$})}
\biggr)$$
$$= \sum_{j=1}^\mu j! (-1)^j \frac{1}{g^{j+1}}\biggl(
\sum_{\kvm} C_{\mbox{\footnotesize \boldmath$k$}}^\mu\,
g^{(\mbox{\footnotesize \boldmath$k$})}
\biggr).$$
We may assume that $||g||_\infty\leq 1$. 
Since the derivatives of $g$ are bounded up to the order $m$, we conclude that on 
$K^\circ\setminus Z(g)$
$$|h^{(\mu)}| \leq C \left|\frac{1}{g}\right|^{\mu+1}.$$
By Leibniz's formula
\begin{equation}\label{kompo}
(f^{m+2}\cdot h)^{(n)}=\sum_{\ell=0}^n {n\choose \ell} h^{(n-\ell)} (f^{m+2})^{(\ell)}.
\end{equation}
Next we have to estimate the derivatives of $f^{m+2}$. Let $p(z)=z^{m+2}$.
Using a second time Theorem \ref{faa}, we obtain for $1\leq \ell\leq m$, 
$$
(f^{m+2})^{(\ell)}= (p\circ f)^{(\ell)}= \sum_{j=1}^\ell (p^{(j)}\circ f) \biggl(
\sum_{\kvl} C_{\mbox{\footnotesize \boldmath$k$}}^\ell\,
f^{(\mbox{\footnotesize \boldmath$k$})}
\biggr).$$
Since $$(z^{m+2})^{(j)}= (m+2)\dots(m+2-j+1)\, z^{m+2-j},$$
we estimate as follows (note that $|f|\leq |g|\leq  1$):
$$
\left|(f^{m+2})^{(\ell)}\right|\leq \sum_{j=1}^\ell C_j |f|^{m+2-j}\leq \tilde C \,|f|^{m+2-\ell}.
$$
In view of the assumption $|f|\leq |g|$, equation \zit{kompo} then yields
$$\left|\left(\frac{f^{m+2}}{g}\right)^{(n)}\right| \leq \tilde C\,  \sum_{\ell=0}^n {n\choose \ell}
\left|\frac{1}{g}\right| ^{n-\ell+1}\;|f|^{m+2-\ell}
$$
$$\buildrel\leq_{|f|\leq |g|}^{}\tilde C\, \sum_{\ell=0}^n  {n\choose \ell} |g|^{m+1-n}=\kappa |g|^{m+1-n}.
$$
\end{proof}

A similar proof applied to the mixed partial derivatives  of $f,g\in C^m(K)$ yields an analogous
result.
\begin{theorem}\label{divic}
Let $f,g\in C^m(K)$ and suppose that $|f|\leq |g|$.
Then, for every $n\in\N$ with  $0\leq n\leq m$, the following estimate
 holds on $K^\circ\setminus Z(g)$: 
$$D_n\left(\frac{f^{m+2}}{g}\right)\leq C\, |g|^{m+1-n},$$
where $D_n=(\partial x)^{j_1}\;(\partial y)^{j_2}$ with $j_1+j_2=n$.
\end{theorem}

\begin{theorem} 
Let $K\ss\C$ be a compact set. Then the following assertions hold:
\begin{enumerate}
\item[a)] If $f,g\in A^m(K)$ satisfy $|f|\leq |g|$, then $f^{m+2}\in  I_{A^m(K)}(g)$.
\item[b)] If $f,g\in C^m(K)$ satisfy $|f|\leq |g|$ and if $Z(g)$ has no cluster points in $K^\circ$,
then \\ $f^{m+2}\in  I_{C^m(K)}(g)$.
\end{enumerate}
\end{theorem}

\begin{proof}
a) Due to holomorphy, the quotient $f^{m+2}/g$ is holomorphic at every isolated zero $z_0$ of $g$
in $K^\circ$,
since $m(f,z_0)\geq m(g,z_0)$, where $m(f,z_0)$ denotes the multiplicity of the zero $z_0$. 
Let 
$$h(z)=\begin{cases} \frac{f^{m+2}(z)}{g(z)} & \text{if $z\in K\setminus Z(g)$}\\
                                     0 & \text{if $z\in Z(g)$.}
\end{cases}
$$
Since $h\equiv 0$ on $Z(g)^\circ$, the fact that $\partial Z(g)^\circ\ss \partial K$ implies
that $h$ is holomorphic on 
$$K^\circ\buildrel=_{}^{*}(K\setminus Z(g)^\circ)^\circ \union Z(g)^\circ.$$
(* was shown in \cite[p. 2219]{moru}).
We may assume that $|g|\leq 1$ on $K$.
By Theorem \ref{divi}, $$|\left(f^{m+2}/g\right)^{(j)}|\leq |g|$$ on $K^\circ\setminus Z(g)$ for every $0\leq j\leq m$. Since $f,g\in A(K)^m$, the derivatives (a priori defined only on $K^\circ$)
 are continuously  extendable to $K$. We denote them by the usual symbol $f^{(j)}$ etc.
Thus $|h^{(j)}|\leq |g|$ on $K\setminus Z(g)$ and so   $h^{(j)}|_{K\setminus Z(g)}$ admits a continuous extension to $K$. \\
b) Similar proof, since we assumed that  the zeros of $f$ and $g$ are isolated.
\end{proof}

\section{The $f^N$-problem in $A(K)$.}

Here we present some sufficient conditions on the generators $f_j$ that guarantee
that $|g|\leq \sum_{j=1}^N |f_j|$ implies that  $g^N\in I_{A(K)}(f_1,\dots, f_n)$
for some $N\in \N$. 

 \begin{lemma}\label{power4}
 Let $A=C^1(K)$ or $C(K)\inter C^1(K^\circ)$.
If  $g,f_j\in  A$ satisfy  $|g|\leq |\bs f|$, where $\bs f=(f_1,\dots, f_n)$,
 then $\left(g^4/|\bs f|^2\right)|_{K\setminus Z(|\bs f|^2)} $ admits an extension to an element in $A$
 whenever $|\bs f|$ has only isolated zeros in $K^\circ$. 
  The power $4$ is best possible (within $\N$). 
 In particular,   $g^4\in I_{A}(f_1,\dots, f_n)$.

\end{lemma}
 \begin{proof}
 To show that $4$ is best possible, consider the function $f(z)=\ov z$ and $g(z)=z$.
 Then $z^3/|z|^2=z^2/\ov z$ is not differentiable at $0$ (see the example in Theorem \ref{powers}(b)).
 Now if $D=\partial$ or $\dbar$, then \footnote{Here $\bs a\cdot \bs b$ means the 
 scalar product of the row-vectors $\bs a$ and $\bs b$.}
 on $K^\circ \setminus Z(|\bs f|^2)$
 \begin{eqnarray}
 D\left(\frac{g^4}{|\bs f|^2}   \right)&=&\frac{4|\bs f|^2 g^3 Dg -g^4 (\bs f\cdot D\ov {\bs f}+
 D \bs f \cdot \ov {\bs f})}{|\bs f|^4}\\
 &=& 4g \left( \Bigl(\frac{g}{|\bs f|}\Bigr)^2 Dg \right)- \Bigl( \frac{g}{|\bs f|}\Bigr)^4 
\left( \bs f\cdot D\ov {\bs f}+D \bs f \cdot \ov {\bs f}\right).
  \end{eqnarray}
  If $|\bs f(z)|=0$, then $g(z)=0$ and so the boundedness of the term $g/|\bs f|$ yields the 
  continuous extensions of $ D\left(\frac{g^4}{|\bs f|^2}\right)$ with value 0.
  Since the zeros of $|\bs f|$ are isolated within $K^\circ$, we get from the mean value theorem
  of the differential calculus, that both partial derivatives exist at those zeros  and coincide
  with these extensions.  Hence $(g^4 /|\bs f|^2)|_{K\setminus Z(\bs f)}$ admits the desired extension.
\end{proof}
 
 By a similar proof we have
 
\begin{lemma}\label{power7}
 Let $A=C^1(K)$ or $C(K)\inter C^1(K^\circ)$.
If  $g,f_j\in  A$ satisfy  $|g|^2\leq |\bs f|$, where $\bs f=(f_1,\dots, f_n)$,
 then $\left(g^7/|\bs f|^2\right)|_{K\setminus Z(|\bs f|^2)} $ admits an extension to an element in $A$
  whenever $|\bs f|$ has only isolated zeros in $K^\circ$. 
   In particular,   $g^7\in I_{A}(f_1,\dots, f_n)$.
\end{lemma}
We don't know whether   the power 7 is optimal.
 
Now we use again the convenient  matricial notation from section  \ref{2}.

\begin{proposition}\label{power5}
 Suppose that  $\bs f=(f_1,\dots,f_n)\in A(K)^n$    and that  $g\in A(K)$ satisfies  $|g|\leq \sum_{j=1}^n |f_j|$. 
 If there is a solution $\bs x$ in  $C_{\dbar,1}(K)$ to $\bs x \bs f^t=g$, then
 there exists $\bs u=(u_1,\dots, u_n)\in A(K)^n$ such that $\bs u \bs f^t= g^5$
  whenever $|\bs f|$ has only isolated zeros in $K^\circ$ and
  there exists $\bs v=(v_1,\dots, v_n)\in A(K)^n$ such that $\bs v \bs f^t= g^6$
  whenever $\bs f \in A(K)^n$ is arbitrary.

 \end{proposition}
 
 \begin{proof}
 We first suppose that $Z(|\bs f|)$ does not admit a cluster point within $K^\circ$. 
 As in Theorem \ref{corona},  we consider on $K\setminus Z(|\bs f|)$ the matrix
 $$F=\left( \left(\dbar{\bs x}^t\cdot \ov{\bs f}\right)^t- \dbar{\bs x}^t\cdot \ov{\bs f}\right)
\frac{1}{|\bs f|^2}.$$

Using the facts that  $\bs x\in C_{\dbar,1}(K)^n$ and   
$g^4\;/\;|\bs f|^2\in C(K)\inter C^1(K^\circ)$ we conclude from Lemma \ref{power4} that
  $F g^4$ extends to  an antisymmetric  matrix over $C(K)\inter C^1(K^\circ)$. Thus,  by Theorem \ref{cr},
the system $\dbar H= F g^4$ admits  a matrix  solution $H$  over $C(K)\inter C^1(K^\circ)$. 
Note that $H$ can be chosen to be   antisymmetric, too.  Now let
$$\bs u= g^4\bs x -\bs f H.$$
Then $\bs u\in C(K)^n\inter  \bigl(C^1(K^\circ)\bigr)^n$. Moreover, 
  on $K^\circ\setminus Z(|\bs f|)$,  $\dbar \bs u=0$ because
$$
\dbar \bs u=g^4\dbar {\bs x}-  \bs f\cdot \dbar H=g^4\dbar{\bs x}-\bs f\cdot\left( \ov{\bs f}^t \cdot \dbar{\bs x} -\dbar{\bs x}^t \cdot \ov{\bs f}\right)\frac{g^4}{|\bs f|^2}$$
$$=g^4\; \frac{(\bs f \cdot \dbar {\bs x}^t )\; \ov{\bs f}}{|\bs f|^2}=
 g^4\;  \frac{\bigl(\dbar( \bs f \cdot\bs x^t)\bigr) \;\ov{\bs f}}{|\bs f|^2}=
 g^4\;  \frac{(\dbar g)\, \ov{\bs f}}{|\bs f|^2}=\bs 0.
$$
Since it is assumed that every point in $Z(|\bs f|)\inter K^\circ$ is an isolated point, the continuity
of $\dbar \bs u$ implies that $\dbar \bs u=0$ on $K^\circ$. Hence $\bs u\in A(K)^n$.
By antisymmetry, \footnote{$\bs f H \bs f^t=(\bs f H \bs f^t)^t=
\bs f H^t \bs f^t=-\bs f H \bs f^t$}  $\bs f H \bs f^t=0$ and so
$$\bs u \bs f^t= g^4 \bs x \bs f^t -\bs f H \bs f^t= g^4 g=g^5.$$\\

Now let $\bs f\in A(K)^n$ be arbitrary. If $|\bs f|$ is identically zero, then nothing is to prove.
 Hence we may assume that $S:=K\setminus Z(|\bs f|)^\circ\not=\emp$.  By the first case,
 there is $\bs u\in A(S)^n$ such that $\bs u\bs f^t =g^5$ on $S$. 
 Since $\partial (Z(|f|)^\circ)\ss\partial K$,  we again have
 $$S^\circ\union   Z(|f|)^\circ= K^\circ.$$
 Hence the vector-valued function
 $$\bs v(z)=\begin{cases} g(z)\bs u(z) & \text{ if $z\in S$}\\
                                           \bs 0 &  \text{ if $z\in K\setminus S$}
 \end{cases}
 $$
 is well defined, continuous on $K$ and each of its coordinates is holomorphic on $K^\circ$.
 It easily follows that  $\bs v \bs f^t=g^6$.
  \end{proof}
\bigskip

According to \cite{mo84}, an ideal $I$ in  a uniform algebra $A$ is said to have the {\it Forelli
 property} (say $I\in \F$),  if there exists $f\in I$ with $Z(f)=Z(I):=\Inter_{h\in I} Z(h)$. 
It is obvious that  in $C(K)$ every  finitely generated ideal has  the Forelli-property;
 just consider the function 
 $$f=\sum_{j=1}^n |f_j|^2=\sum_{j=1}^n \ov f_j  f_j\in I_{C(K)}(f_1,\dots,f_n).$$
  It was shown in \cite{mo84}
 that there exist finitely generated ideals in the disk-algebra that do not have this property.  
 A natural question, therefore, is whether  $C_{\dbar,1}(K)$ has the Forelli-property.
 A sufficient condition for $I=I(f_1,\dots,f_n)$  to belong to $\F$ is that there exists
 $h\in I$ such that \footnote{Instead of 2 we may of course take any power $\alpha>0$.}
  $$|h|\geq \sum_{j=1}^n|f_j|^2.$$

  \begin{theorem}
 Let $\bs f=(f_1,\dots,f_n)\in A(K)^n$ and suppose that $|\bs f|$ has only isolated zeros in $K^\circ$. We assume that there are $h_j\in C_{\dbar,1}(K)$ such that
 $$\bigl|\sum_{j=1}^n h_jf_j\bigr|\geq  \sum_{j=1}^n|f_j|^2.$$
 Then for every $g\in A(K)$ satisfying $|g|\leq \sum_{j=1}^n|f_j|$
 we have $g^{12}\in I_{A(K)}(f_1,\dots,f_n)$.
  \end{theorem}
 \begin{proof}
 Let $h=\sum_{j=1}^n h_jf_j$. Then $|g|^2\leq  n\, \sum_{j=1}^n |f_j|^2 \leq n \,|h|$. 
 By Proposition \ref{powers}, there is 
 $k\in C_{\dbar,1}(K)$ such that $g^8=kh$.
  Hence 
 $$g^8=\sum_{j=1}^n (kh_j) f_j\in I_{ C_{\dbar,1}(K)}(f_1,\dots,f_n).$$
 
  Since $|g|\leq \sqrt n\,|\bs f|$ implies that  $g^4/|\bs f|^2 \in C(K)\inter C^1(K^\circ)$ 
  (Proposition \ref{power4}), 
  we obtain in a similar manner as in Proposition \ref{power5}
  that $g^{12}\in I_{A(K)}(f_1,\dots, f_n)$ (just consider the data $\dbar H=F g^4$,  
  $\bs x=(kh_1, \dots, kh_n)$, 
   $\bs u=g^4\bs x -\bs f H$,
  and use the fact that  $\bs x\bs f^t=g^8$ in order to get $u\bs f^t=g^{12}$).
 \end{proof}
 
\section{The $f^N$-problem in $C(K)$ and $C^1(K)$}

\begin{proposition}
Let $h,f_j\in C(K)$ satisfy $|h|\leq \sum_{j=1}^n |f_j|$. Then
 $h^2\in I_{C(K)}(f_1,\dots,f_n)$. Within $\N$, the constant 2 is best possible.
\end{proposition}
\begin{proof}
Consider on $K\setminus \Inter_{k=1}^n Z(f_k)$ the functions 
$$q_j=\frac{h \ov f_j}{\sum_{k=1}^n |f_k|^2}.$$
Then, by Cauchy-Schwarz,  
$$|q_j|\leq \frac{\left(\sum_{k=1}^n|f_k |\right)^2}{|\bs f|^2}\leq n.$$
Hence $q_j$ is bounded on $K\setminus \Inter_{k=1}^n Z(f_k)$.
It is then clear that  $g_j:=hq_j\in C(K)$ and that
$$\sum_{j=1}^n g_j f_j=h^2.$$
By Theorem \ref{powers}(a), the power $n=2$ is best possible.
\end{proof}

\begin{proposition}
 Let $h,f_j\in C^1(K)$ satisfy $|h|\leq \sum_{j=1}^n |f_j|$. If we suppose that
 $\Inter_{j=1}^n Z(f_j)\inter K^\circ$ is discrete, then
 $h^3\in I_{C^1(K)}(f_1,\dots,f_n)$. Within $\N$, the constant 3 is best possible.
\end{proposition}
\begin{proof}
The example $h(z)=z$ and $f(z)=\ov z$ in Proposition \ref{powers} (b) shows that 3 is best possible.
On $K\setminus \Inter_{j=1}^n  Z(f_j)$, let
$$q_j=\frac{\ov f_j h^3}{\sum_{k=1}^n |f_k|^2}.$$
We claim that $q_j$ has an extension to a function in $C^1(K)$. 
Let $\bs f=(f_1,\dots,f_n)$.
Since outside $\Inter_{j=1}^n  Z(f_j)$
$$|q_j|\leq  \frac{|\bs f| |h^3|}{|\bs f|^2}=\left(\frac{|h|}{|\bs f|}\right) |h|^2\leq C\, |\bs f|^2,$$
we immediately see that $q_j$ can be  continuously extended to $K$ with value 0 on $Z(|\bs f|)$.
Now, if $D=\partial$ or $\dbar$, then on  
$K^\circ \setminus \Inter_{j=1}^n  Z(f_j)=K^\circ\setminus Z(|\bs f|)$
$$D q_j=\frac{|\bs f|^2\bigl( (D\ov f_j)h^3 +3\ov f_j (Dh) h^2\bigr)
-\ov f_j h^3(\bs f D\ov{\bs f} +\ov{\bs f} D\bs f)}{|\bs f|^4}.
$$
Using that the derivatives are continuous and that $\max \{|h|, |f_j|\}\leq \kappa \,|\bs f|$,
we obtain  constants $C_j$  such that
$$|Dq_j| \leq \left(\frac{|h|}{|\bs f|}\right)^2|D\ov f_j| \; |h| +
3 |\bs f| \;|Dh| \left(\frac{|h|}{|\bs f|}\right)^2+
\frac{|\bs f| \;|h|^3 \;|\bs f|( |D\bs f|+|D\ov{\bs f}|)}{|\bs f|^4}
$$
$$ \leq C_1|\bs f|+  C_2 |\bs f | + |\bs f| \; (|D\bs f|+|D\ov{\bs f}|)\leq C_3|\bs f|.
$$
Thus the partial derivatives admit a continuous extension to $K$. Since
at an isolated zero of $|\bs f|$ the partial derivatives of $q_j$ exist by the mean value
theorem in differential calculus,  we are done.
\end{proof}

\section{Questions}

In this final section we would like to ask several questions and present a series of problems. \\
\begin{enumerate}
\item [1)]  Let $f,f_1,\dots,f_n\in A(\ov\D)$ satisfy $|f|\leq \sum_{j=1}^n |f_j|$. Is 
$f^3\in I_{A(\ov\D)}(f_1,\dots, f_n)$?

\item [2)]  Is there a simple example of a triple $(f,f_1,f_2)$ in $A(\D)$ such that $|f|\leq |f_1| +|f_2|$,
but for which $f^2\notin  I_{A(\ov\D)}(f_1, f_2)$?

\item [3)] Is there a simple example of a function $f$, continuous on $\ov\D$, such that the Pompeiu-integral
$$P_f(z)=\dint_{\ov\D} \frac{f(\xi)}{\xi-z}\;d\sigma(\xi)$$
is not continuously differentiable in $\D$?
\end{enumerate}
\medskip

Let $M$ and $S$ be two specific classes of functions on the unit disk. For example $M$ and $S$
coincide with 
$C_b(\D)$, the set of all bounded,  continuous and  complex-valued  functions on $\D$;
$C^1_b(\D)=C_b(\D)\inter C^1(\D)$; or
$C^1_{bb}(\D)$,  the set of all $C^1$-functions $u$ on $\D$ for which $u$ and $\nabla u$
are bounded.
\medskip

 Give necessary and sufficient conditions on $f\in M$ such that the $\dbar$-equation
 $\dbar u=f$ admits a solution  $u\in S$:
 \medskip
 
 \begin{enumerate}
 \item [4)] $f\in C_b^1(\D)$, $u\in C_b^1(\D)$;
 \item  [5)]$f\in C_b^1(\D)$, $u\in C(\ov \D)\inter C^1(\D)$;
 \item [6)]$f\in C_b^1(\D)$, $u\in C(\ov \D)\inter C_{bb}^1(\D)$;
  \item [7)]$f\in C^1(\D)$,  $u\in C^1(\D)$;
 \item [8)]$f\in C^1(\D)$,  $u\in C_b^1(\D)$;
\item [9)] $f\in C^1(\D)$, $u\in C(\ov \D)\inter C^1(\D)$;
 \item [10)]$f\in C^1(\D)$, $u\in C(\ov \D)\inter C_{bb}^1(\D)$;
 \item [11)]$f\in C(\ov \D)$, $u\in C^1(\D)$;
  \item [12)]$f\in C(\ov \D)$, $u\in C_b^1(\D)$;
    \item [13)] $f\in C(\ov \D)$, $u\in C(\ov\D) \inter C^1(\D)$;
     \item [14)]$f\in C(\ov \D)$, $u\in C(\ov \D)\inter C_{bb}^1(\D)$.

\end{enumerate}




\begin{thebibliography}{99}

\bibitem{an} M. Andersson, \emph{Topics in Complex Analysis}, Springer, New York 1997.

\bibitem{ga} {J.B. Garnett}, \emph{Bounded Analytic Functions}, 
Academic Press, New York, 1981.

\bibitem{mo84} R. Mortini, The Forelli problem concerning ideals in the disk algebra $A(\mathbf D)$,
Proc. Amer. Math. Soc. 95 (1985), 261-264.

\bibitem{moelem} R. Mortini, The Fa\`a di Bruno formula revisited,
to appear in Elem. Math.  2013.

\bibitem{moru} R. Mortini, R. Rupp, The Bass stable rank for the real Banach algebra 
$A(K)_{\ssc sym}$, 
J. Funct. Anal. 261 (2011), 2214-2237.

\bibitem{mr} R. Mortini, R. Rupp, A solution to the B\'ezout equation in $A(K)$  without Gelfand theory. Archiv Math. 99 (2012), 49-59.

\bibitem{mw} R. Mortini, B. Wick,  Simultaneous stabilization in $A_\R(\D)$, 
Studia Math. 191 (2009), 223-235.

\bibitem{na} R. Narasimhan \emph{Complex variables in one variable}, Birkh\"auser, Boston,
1985. 

\bibitem{ru} W. Rudin \emph{Functional Analysis}, Second Edition, MacGraw-Hill, 1991.

\end{thebibliography}
 \end{document}